\RequirePackage[hyphens]{url}
\documentclass[a4paper]{ifacconf}
\usepackage[T1]{fontenc}
\usepackage{lmodern}

\usepackage{comment}
\usepackage{enumitem}
\usepackage{natbib}
\usepackage{notation}

\usepackage{figures}
\newcommand{\defining}[1]{\textbf{#1}}

\begin{document}

\begin{frontmatter}

\title{The Dynamic Search for the Minimal Dynamic Extension}
\thanks[footnoteinfo]{This project was not funded or sponsored by any institution or grant. This work has been submitted to IFAC for possible publication.}
\author[First]{Rollen S. D'Souza} 
\address[First]{e-mail: research@rollends.ca.}

\begin{abstract}
  Identifying the dynamic precompensator that renders a nonlinear control system feedback linearizable is a challenging problem.
  Researchers have explored the problem --- dynamic feedback linearization --- and produced existence conditions and constructive procedures for the dynamic precompensator.
  These remain, in general, either computationally expensive or restrictive.
  Treating the challenge as intrinsic, this article views the problem as a search problem over a category.
  Dynamic programming applies and, upon restriction to a finite category, classic search algorithms find the minimal dynamic extension.
  Alternatively, a heuristic aiming towards feedback linearizable systems can be employed to select amongst the infinitely-many extensions.
  This framing provides a distinctive, birds-eye view of the search for the dynamic precompensator.
\end{abstract}


\end{frontmatter}

\section{Introduction}
Identifying whether a nonlinear control system,\hfill
\[
  \dot{x} = f(x, u),
\]
can be transformed into a linear control system in Brunovsk\'y normal form,\hfill
\[
  \dot{\xi} = A\,\xi + B\,\nu,
\]
through the use of a state (\(\xi = \phi(x)\)) and feedback transformation (\(\nu = \beta(x, u)\)) has been well-understood for decades~\citep{isidori_nonlinear_1995}.
A nonlinear control system that has such a transformation is said to be (statically) feedback linearizable.
\cite{hunt_design_1982} and~\cite{jakubczyk_linearization_1980} independently proposed the necessary and sufficient conditions for when a nonlinear control system is feedback linearizable.
Unfortunately, the class of feedback linearizable systems is small.
Even simple, practical control systems --- such as those that model non-holonomic vehicles --- fail to be feedback linearizable.
This motivates the now-extensive study of \emph{dynamically} feedback linearizable systems and \emph{(differentially) flat} systems.
Dynamically feedback linearizable systems are those systems that, under appropriate dynamic precompensation,
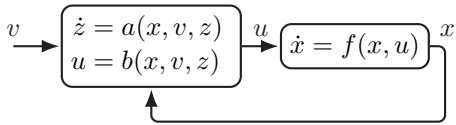
\begin{center}
  \begin{tikzpicture}[x = 1cm, y = 1cm]

    \node[at = {(0, 0)}] (v) {};
    \node[system block, right = 0.5 of v] (Precompensator) {%
      \(%
        \begin{aligned}
          \dot{z} &= a(x, v, z)\\
          u &= b(x, v, z)
        \end{aligned}%
      \)
    };
    \node[system block, right = 0.5 of Precompensator] (System) {%
      \(\dot{x} = f(x, u)\)%
    };

    \draw[system path, - Latex]
      (v.base) node[above] {\(v\)} -- (Precompensator.west);
    \draw[system path, - Latex]
      (Precompensator.east) -- node[above, midway] {\(u\)}(System.west);
    \draw[system path, - Latex]
      (System.east) 
        node[above right] {\(x\)}
        -- 
      +(0.2, 0)
        --
      +(0.2, -1)
        -|
      (Precompensator.south);

  \end{tikzpicture}
\end{center}
are feedback linearizable~\citep{charlet_dynamic_1989}.
Flat systems are similar in that they require an equivalence between the original system \((x, u)\) and the dynamics of an output and a finite number of its derivatives~\citep{levine_analysis_2009}.
\cite{levine_equivalence_2007} demonstrated the equivalence between flat and dynamically feedback linearizable systems.
Conditions upon which a system is dynamically feedback linearizable are known since the work of~\cite{guay_condition_1997}, but finding the required precompensator efficiently remains an open problem.
By restricting to two-input systems,~\cite{gstottner_necessary_2023} demonstrated a method to construct the flat outputs.
The conditions for flatness even in this case are complex, suggesting that the general conditions for flatness are computationally complex.
What if no computationally reasonable method exists for the multi-input, dynamic feedback linearizability problem?

Surrendering to this possibility, it is worth taking a hint from computer science where computationally expensive searches are made practical through heuristics.
Searching for differential geometric objects to achieve some end is not new.
The process of \emph{prolongation} was leveraged to great success by \'E. Cartan for bringing differential systems into integrability~\citep{malgrange_sur_2017, bryant_exterior_1991}.
Similarly, states added to a nonlinear control system that are derivatives of the input can result in a feedback linearizable system.
This process is called \emph{dynamic extension} and it mirrors prolongation.
When a sequence of dynamic extensions leads to a feedback linearizable system, the original system is dynamically feedback linearizable.
The reverse is not necessarily the case~\citep{nicolau_dynamic_2025}.
Systems that are dynamically feedback linearizable via a sequence of one-fold prolongations fall in the LSOP class.
The incremental construction of a precompensator through one-fold prolongations is popular and a number of algorithms have been proposed to arrive at a dynamic precompensator~\citep{battilotti_constructive_2004, astolfi_geometric_2008, levine_differential_2025-1, nicolau_dynamic_2025}.

Dynamic feedback linearization is seen as a search problem in this broad vein.
Can we search the space of dynamic extensions and incrementally build the minimal dynamic extension to feedback linearizability?
This article frames the problem in a manner that lends itself to finding the minimal dynamic extension efficiently when restricted to a subclass of the LSOP class.
This subclass is explicitly encoded through a formal characterization of the dynamic extension step in Section~\ref{section:geoemetry}.
Section~\ref{section:dynamic-programming} then takes a detour to discuss dynamic programming and its application to mathematical categories.
The prescribed subclass of systems forms a category where the arrows are precisely those projections tied to a dynamic extension.
A dynamic search prunes the category of non-optimal arrows producing, as a result, a category that enumerates all the optimal arrows to the original system.
If at least one arrow comes from a feedback linearizable system, then the system is dynamically feedback linearizable.
Section~\ref{section:deds} puts these pieces together as-described and demonstrates the approach, in part, through simple examples.

\section{Notation}
If \(x\) \(\in\) \(\Reals^n,\) \(x^i\) denotes the \(i\)th component.
Given a map \(F:\) \(\Reals^n\) \(\to\) \(\Reals^m,\) the map \(F^i:\) \(\Reals^n\) \(\to\) \(\Reals\) denotes the \(i\)th component function.
An \(n\)-dimensional manifold \(\Manifold{M}\) \(\subset\) \(\Reals^n\) is described locally through a coordinate chart \((\OpenSet{U};\) \((x^1,\) \(\ldots,\) \(x^n))\) where \(\OpenSet{U}\) \(\subseteq\) \(\Manifold{M}\) is an open subset and \(x^1,\) \(\ldots,\) \(x^n\) are the coordinates.
Let \(p\) \(\in\) \(\OpenSet{U}\) \(\subseteq\) \(\Manifold{M}.\)
The tangent space \(\TangentSpace{\Manifold{M}}{p}\) at \(p\) is spanned by vectors tangent to the manifold at \(p.\)
It can be bundled over \(\Manifold{M}\) to form the tangent bundle \(\TangentBundle{\Manifold{M}}.\)
The cotangent space \(\CotangentSpace{\Manifold{M}}{p}\) at \(p\) is spanned by covectors (one-forms) \(dx^1_p,\) \(\ldots,\) \(dx^n_p : \TangentSpace{\Manifold{M}}{p} \to \Reals.\)
The cotangent space can be bundled together as well to form the cotangent bundle \(\CotangentBundle{\Manifold{M}}.\)
The set of smooth sections of the cotangent bundle is denoted \(\CoTangentSections{\Manifold{M}}\) whose elements \(\omega\in \CoTangentSections{\Manifold{M}}\) are precisely the smooth one-forms of \(\Manifold{M}.\)
For \(k > 1,\) the space of \(k\)-forms at \(p\) is denoted \(\KFormSpaceAt{\Manifold{M}}{p}{k}.\)
Smooth \(k\)-forms are constructed in the same manner as smooth one-forms.
These combine to form the set of all smooth forms \(\FormSections{\Manifold{M}}\) on \(\Manifold{M}.\)
Given \(\beta\) \(\in\) \(\SmoothSections{\KFormSpaceOf{\Manifold{M}}{k}},\) the exterior derivative is \(d\beta\) \(\in\) \(\SmoothSections{\KFormSpaceOf{\Manifold{M}}{k+1}}.\)
If \(\alpha\) \(\in\) \(\SmoothSections{\KFormSpaceOf{\Manifold{M}}{\ell}},\) then \(\alpha\) \(\wedge\) \(\beta\) \(\in\) \(\SmoothSections{\KFormSpaceOf{\Manifold{M}}{k+\ell}}\) is known as their wedge product.
Given \(\beta^1,\) \(\ldots,\) \(\beta^k \in \CoTangentSections{\Manifold{M}},\) an ideal \(\Ideal{I}\) over the (wedge product) algebra of smooth forms may be generated with the notation \(\langle \beta^1,\) \(\ldots,\) \(\beta^k \rangle\) \(\subseteq\) \(\FormSections{\Manifold{M}}.\)
An ideal \(\Ideal{I}\) is differentially closed when \(d\omega \in \Ideal{I}\) for all \(\omega \in \Ideal{I}.\)
The Lie derivative of a form \(\omega\) along a vector field \(f \in \TangentSections{\Manifold{M}}\) is \(\LieDerivative{f}\omega\) with \(k\) repeated applications denoted \(\LieDerivative{f}^k\omega.\)

\section{The Geometry of Dynamic Extension}
\label{section:geoemetry}
A formal definition of dynamic extension is required to characterize the class of systems to search.
The nonlinear control system is viewed as an exterior differential system (i.e. a system of differential constraints)~\citep{guay_condition_1997, dsouza_algorithm_2023}.
Dynamic extension is then seen as the addition of constraints arising through Lie derivatives.

\subsection{Feedback Linearization Theory}
Let \(x \in \Manifold{X}\) \(=\) \(\Reals^n\) be the state vector of \(n\) states and \(u \in \Manifold{U}\) \(=\) \(\Reals^m\) be the control input vector of \(m\) inputs.
Suppose the dynamics of \(x\) are governed by the nonlinear control system,
\begin{equation}
  \label{eqn:dynamics}
  \dot{x} = f(x, u),
\end{equation}
where \(f: \Manifold{X}\times\Manifold{U} \to \TangentBundle{X}\) is smooth.
Equivalently, the system may be viewed as an exterior differential system \(\Ideal{I}^{(0)}\) \(\subseteq\) \(\FormSections{\Manifold{M}}\) on the extended manifold \(\Manifold{M}\) \(=\) \(\Reals\) \(\times\) \(\Manifold{X}\) \(\times\) \(\Manifold{U}\) generated by a set of smooth one-forms,
\begin{equation}
  \label{eqn:eds}
  \Ideal{I}^{(0)}
    \defineas 
      \langle
        dx^1 - f^1(x, u)\,dt,
        \ldots,
        dx^n - f^n(x, u)\,dt
      \rangle.
\end{equation}
The pair \((\Manifold{M}, \Ideal{I}^{(0)})\) is called a nonlinear control system.
Tied to any nonlinear control system is the \emph{derived flag},
\[
  \Ideal{I}^{(0)} \supseteq \Ideal{I}^{(1)} \supseteq \cdots \supseteq \Ideal{I}^{(n)} = \Ideal{I}^{(n+1)},
\]
defined by the recursive formula,
\[
  \Ideal{I}^{(k)} = \{
    \omega :
    d\omega \in \Ideal{I}^{(k-1)}
  \},
  \quad
  k \geq 1.
\]

The derived flag is used --- in its augmented form \(\langle \Ideal{I}^{(k)}, dt \rangle\) --- to identify how inputs drive smooth functions of the state.
Recall that the search for a feedback transformation that renders the system dynamics~\eqref{eqn:dynamics} linear is equivalent to the search for an output function \(h: \Manifold{X} \to \Reals^m\) that yields a special property known as vector relative degree;
this was a view endorsed by~\cite{isidori_nonlinear_1995}.
The definition of (vector) relative degree in its dual formulation follows.
%
\begin{defn}[Vector Relative Degree]
  \label{def:vector-relative-degree}
  A smooth function of the state \(h = (h^1,\ldots, h^\rho): \Manifold{X} \to \Reals^\rho\) is said to have \defining{vector relative degree \((\kappa^1, \ldots, \kappa^\rho)\)} \(\geq\) \(0\) at a point \(x_0 \in \Manifold{X}\) if, on an open neighbourhood of \(x_0,\)
  \[
    \LieDerivative{f}^{j} dh^i \wedge dx^{1} \wedge \cdots dx^{n} = 0, \qquad 1\leq i \leq \rho, 0 \leq j < \kappa^i,
  \]
  and, at \(x_0,\)
  \[
    (\LieDerivative{f}^{\kappa^1}dh^1)_{x_0} \wedge \cdots \wedge (\LieDerivative{f}^{\kappa^\rho}dh^\rho)_{x_0} \wedge
    dx^{1}_{x_0} \wedge \cdots \wedge dx^{n}_{x_0} \neq 0.
  \]
  It is said to have \defining{uniform} vector relative degree if \(\kappa^1 = \cdots = \kappa^\rho.\)
\end{defn}
By Definition~\ref{def:vector-relative-degree}, an output \(h\) has relative degree zero if, and only if,
\[
    dh_{x_0} \wedge dx^{1}_{x_0} \wedge \cdots \wedge dx^{n}_{x_0} \neq 0.
\]
This non-standard notion has use in defining dynamic extensions:
any dynamic extension is tied to some output of relative degree zero.

The derived flag is used to identify those output functions that yield (vector) relative degree.
Outputs of uniform vector relative degree \(\kappa\) must live in the ideal \(\langle \Ideal{I}^{(\kappa-1)}, dt \rangle\) but not live in the ideal \(\langle \Ideal{I}^{(\kappa)}, dt \rangle.\)
Since determining whether a nonlinear control system~\eqref{eqn:eds} is feedback linearizable is equivalent to finding an output of appropriate relative degree, it comes as no surprise that the derived flag structure determines whether a system is feedback linearizable.
The standard dual conditions for feedback linearizability --- assuming the regularity of all objects --- are the involutivity condition,
\begin{equation}
  \tag{Inv}
  \label{eqn:involutivity}
  \langle \Ideal{I}^{(k)}, dt \rangle \subseteq \langle \Ideal{I}^{(k)}, dt \rangle^{(\infty)},
  \quad 0 \leq k \leq n,
\end{equation}
and the controllability condition,
\begin{equation}
  \tag{Con}
  \label{eqn:controllability}
  \langle \Ideal{I}^{(n)}, dt \rangle = \langle dt \rangle.
\end{equation}
All systems in this article are assumed to satisfy the controllability condition~\eqref{eqn:controllability} as this is necessary for dynamic feedback linearization.
It is the failure of the involutivity condition~\eqref{eqn:involutivity} we aim to resolve through dynamic extension.

\subsection{Regular Zero Dynamics Foliations}
To every output \(h:\) \(\Manifold{M}\) \(\to\) \(\Reals^\rho,\) we may associate a zero dynamics manifold \(\Manifold{Z}\) which is the largest controlled-invariant set where the output vanishes.
It characterizes the dynamics of the system when the output is zeroed.
When the output \(h\) has (positive) vector relative degree \(\kappa,\) the zero dynamics manifold has a simple structure where \(\Manifold{Z} = \{ x \in \Manifold{X} : h(x)\) \(=\) \(\cdots\) \(=\) \(\LieDerivative{f}^{\kappa^{\rho}-1} h(x) = 0 \}.\)
All outputs with relative degree induce a zero dynamics manifold, but not all zero dynamics manifolds are induced by an output with relative degree.
For this reason,~\cite{dsouza_algorithm_2023} introduced the notion of a \emph{regular} zero dynamics manifold.
Instead of placing an emphasis on a single manifold, we instead consider a \emph{family} of manifolds induced by the level sets of a smooth function.
The level sets of a smooth function \(h: \OpenSet{U} \to \Reals^\rho\) near a regular point \(x_0\) form an \emph{\((n-\rho)\)-dimensional foliation}:
the disjoint union of \((n-\rho)\)-dimensional, immersed submanifolds (leaf) of \(\OpenSet{U}\) that cover the space and for which a coordinate system exists where they are ``flat''~\citep[pp. 501]{lee_introduction_2012}.
Like in~\cite{dsouza_algorithm_2023}, these smooth functions are required to have vector relative degree.
This defines a regular zero dynamics foliation.
\begin{defn}[Regular Zero Dynamics Foliation]
  A foliation \(\Foliation{Z}\) is said to be a \defining{regular zero dynamics foliation (of type \((\rho, \kappa)\))} if there exists a smooth function of the state \(h: \Manifold{X} \to \Reals^\rho\) that has vector relative degree \(\kappa\) and whose level sets are the leaves of the foliation.
\end{defn}
It is straightforward to define a regular zero dynamics foliation from an output with relative degree.
On the other hand, how do we know that an arbitrary foliation is generated by an output with relative degree?

Given \emph{any} foliation \(\Foliation{Z},\) we can construct the differential ideal,
\[
  \mathfrak{I}(\Foliation{Z})
    \defineas
      \{
        \omega \in \FormSections{\Manifold{M}}
        \colon
        \left.\omega\right|_\Manifold{N} = 0, \forall \Manifold{N} \in \Foliation{Z}
      \}.
\]
To any differential ideal \(\Ideal{K}\), Frobenius's Theorem may be used to construct the foliation \(\mathfrak{F}(\Ideal{K})\) that generates it~\citep{bryant_exterior_1991}.
Clearly \(\mathfrak{F}(\mathfrak{I}(\Foliation{Z})) = \Foliation{Z}\) and \(\mathfrak{I}(\mathfrak{F}(\Ideal{K})) = \Ideal{K}.\)
Given any two foliations \(\Foliation{Z}_1\) and \(\Foliation{Z}_2\), define their \emph{meet} by,
\begin{equation}
  \label{eqn:meet}
  \Foliation{Z}_1 \wedge \Foliation{Z}_2
    \defineas
      \mathfrak{F}\left(
        \mathfrak{I}(\Foliation{Z}_1)
        +
        \mathfrak{I}(\Foliation{Z}_2)
      \right).
\end{equation}
When the meet \(\Foliation{Z}_1 \wedge \Foliation{Z}_2\) is a regular zero dynamics foliation, \(\Foliation{Z}_1 \wedge \Foliation{Z}_2\) is the largest regular zero dynamics foliation that simultaneously foliates leaves of \(\Foliation{Z}_1\) and \(\Foliation{Z}_2.\)
A partial order may be induced by the meet operation.
Define the binary relation,
\[
  \Foliation{Z}_1 \leq \Foliation{Z}_2
  \quad \mathrm{iff} \quad 
  \Foliation{Z}_1 \wedge \Foliation{Z}_2 = \Foliation{Z}_1.
\]
Clearly if \(\Foliation{Z}_1 \leq \Foliation{Z}_2\) then \(\Dimension\Foliation{Z}_1 \leq \Dimension\Foliation{Z}_2.\)

Define \(\mathfrak{Z}(\Ideal{K})\) as the largest foliation where leaves are rendered invariant under the dynamics \(f: \Manifold{X}\times\Manifold{U} \to \TangentBundle{X}\) of the nonlinear control system \(\Ideal{I}^{(0)}.\)
Formally,
\begin{equation}
  \label{eqn:controlled-invariance}
  \LieDerivative{f} \left[
    \mathfrak{I}(\mathfrak{Z}(\Ideal{K}))
    \cap
    \langle \Ideal{I}^{(0)}, dt \rangle
  \right]
  \subseteq
    \mathfrak{I}(\mathfrak{Z}(\Ideal{K})).
\end{equation}
This is a controlled-invariance requirement.
When \(\Ideal{K}\) is differentially closed, \(\mathfrak{Z}(\Ideal{K})\) is computed using repeated Lie derivatives of the exact generators along \(f\) until convergence.

Given a foliation \(\Foliation{Z},\) the ideal \(\mathfrak{I}(\Foliation{Z})\) is used to identify whether or not \(\Foliation{Z}\) is a regular zero dynamics foliation.
Asking whether \(\Foliation{Z}\) is a regular zero dynamics foliation is approximately equivalent to demanding that each leaf of \(\Foliation{Z}\) is transverse feedback linearizable.
A minor modification of the algorithm presented by~\cite{dsouza_algorithm_2023} produces the outputs that define the regular zero dynamics foliation.
\begin{prop}
  \label{prop:reg-zero-dyn-fol}
  A \(d\)-dimensional foliation \(\Foliation{Z}\) of \(\OpenSet{V} \subseteq \Manifold{M}\) is a regular zero dynamics foliation (of type \((\rho, \kappa)\)) at \(p_0 = (t_0, u_0, x_0) \in \OpenSet{U}\) on \((\Manifold{M}, \Ideal{I}^{(0)})\) if,
  \begin{enumerate}[label={(\roman*)}]
    \item{
      \label{prop:reg-zero-dyn-fol:invariance}
      \(
        \Foliation{Z} = \mathfrak{Z}(\mathfrak{I}(\Foliation{Z})),
      \)
    }
    \item{
      \label{prop:reg-zero-dyn-fol:regularity}
      \(
        \left\langle \Ideal{I}^{(k)}, dt \right\rangle \cap \mathfrak{I}(\Foliation{Z})
      \)
      is simply, finitely, non-degenerately generated, for every \(0 \leq k \leq d,\)
    }
    \item{
      \label{prop:reg-zero-dyn-fol:involutivity}
      \(
        \left\langle \Ideal{I}^{(k)}, dt \right\rangle \cap \mathfrak{I}(\Foliation{Z})
      \)
      is differentially closed for every \(0 \leq k \leq d,\) and
    }
    \item{
      \label{prop:reg-zero-dyn-fol:controllability}
      \(
        \left\langle \Ideal{I}^{(n-d)}, dt \right\rangle \cap \mathfrak{I}(\Foliation{Z})
        =
        \langle dt \rangle.
      \)
    }
  \end{enumerate}
\end{prop}
\begin{pf}See Appendix~\ref{pf:prop:reg-zero-dyn-fol}\end{pf}
%

The proof of Proposition~\ref{prop:reg-zero-dyn-fol} provides an algorithm that computes the output with suitable vector relative degree at \(x_0\) whose levels sets coincide with leaves of the foliation \(\Foliation{Z}.\)
We can now, without loss of generality, move between regular zero dynamics foliations, their differential ideals, and the output with vector relative degree used to generate them.

\subsection{Extension along a Regular Zero Dynamics Foliation}
Regular zero dynamics foliations provide enough information to define a notion of dynamic extension.
Given a regular zero dynamics foliation of type \((\rho^0, \kappa)\), we can construct an output \(h: \Manifold{M} \to \Reals^{\rho^0}\) that has vector relative degree \(\kappa.\)
This output defines a channel to \(\rho^0\) inputs.
The dynamic extension of this channel is not a pure prolongation:
the derivatives of \(h\) are driven by a function of the input and state.
Write \(h = (h^1, \ldots, h^{\rho_0})\) and define virtual inputs, for every \(1 \leq j \leq \rho_0,\) by setting,
\[
  \LieOperator_{f}^{\kappa_j} h^j = \alpha(x, u) \asdefine
  v^j.
\]
Prolongation of the inputs \(v^j\) constitutes a dynamic extension of the output \(h.\)
This class of extensions is illustrated in Figure~\ref{fig:dsouza-dynamic-extension}.
The dynamics of the components \(\alpha^j\) --- whose right-inverse \(u = \beta(x, z)\) is the preliminary feedback --- are in Brunovsk\'y normal form.
\begin{figure}
  \centering
  \begin{tikzpicture}[x = 1cm, y = 1cm]

    \node[at = {(0, 0)}] (v) {};
    \node[system block, right = 0.5 of v] (Precompensator) {%
      \(Z^j(s) = \frac{1}{s^{\kappa^j}}\,V(s)\)%
    };
    \node[system block, right = 0.5 of Precompensator] (FT) {%
      \(\beta(z, x)\)%
    };
    \node[system block, right = 0.5 of FT] (System) {%
      \(\dot{x} = f(x, u)\)%
    };

    \draw[system path, - Latex]
      (v.base) node[above] {\(v\)} -- (Precompensator.west);
    \draw[system path, - Latex]
      (Precompensator.east) --  node[above, midway] {\(z\)} (FT.west);
    \draw[system path, - Latex]
      (FT.east) -- node[above, midway] {\(u\)}(System.west);
    \draw[system path, - Latex]
      (System.east) 
        node[above right] {\(x\)}
        -- 
      +(0.2, 0)
        --
      +(0.2, -1)
        -|
      (FT.south);

  \end{tikzpicture}
  \caption{
    A system precompensated by outputs extended by chains of integrators with lengths \(\kappa^j.\)
  }
  \label{fig:dsouza-dynamic-extension}
\end{figure}
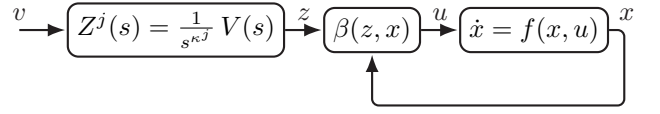
This concept is captured in an extension system: a (lifted) regular nonlinear control system which projects down in just the right way onto a (base) regular nonlinear control system.
\begin{defn}[Extension System]
  \label{def:extension}
  Let \((\Manifold{M}_1, \Ideal{I}^{(0)}_1)\) be a regular nonlinear control system (with time \(t,\) inputs \(u\) and states \(x\)) and take \((\Manifold{M}_2, \Ideal{I}^{(0)}_2)\) to be another regular nonlinear control system (with time \(t,\) inputs \(v\) and states \(y\)).
  The system \((\Manifold{M}_2, \Ideal{I}^{(0)}_2)\) is an \defining{extension system} of the system \((\Manifold{M}_1, \Ideal{I}^{(0)}_1)\) if there exists a smooth submersion \(\pi: \Manifold{M}_2 \to \Manifold{M}_1\) which takes the form 
  \[
    (t, u, x) = \pi(t, v, y) = (t, \beta(v, y), \phi(y))
  \]
  and satisfies,
  \begin{enumerate}[label={(\roman*)}]
    \item{
      \label{def:extension:constraints}
      \(\pi^* \Ideal{I}^{(0)}_1 \subseteq \Ideal{I}^{(0)}_2,\) and
    }
    \item{
      \label{def:extension:extensionmanifold}
      there exists a regular zero dynamics foliation \(\Foliation{Z}_1\) on \((\Manifold{M}_1, \Ideal{I}^{(0)}_1)\) so that
      \[
          \langle \pi^* \FormSections{\Manifold{M}_1}, \mathfrak{I}(\Foliation{Z}_2), dt \rangle = \FormSections{\Manifold{M}_2},
      \]
      where \(\Foliation{Z}_2\) \(=\) \(\mathfrak{Z}(\pi^* \mathfrak{I}(\Foliation{Z}_1))\) is a regular zero dynamics foliation on \((\Manifold{M}_2, \Ideal{I}^{(0)}_2).\)
    }
  \end{enumerate}
\end{defn}
The first condition~\ref{def:extension:constraints} ensures the base system's dynamics are captured in the lifted system \(\Ideal{I}^{(0)}_2\) on \(\Manifold{M}_2.\)
The definition's weight is carried by condition~\ref{def:extension:extensionmanifold} which enforces the existence of an output used to construct the dynamic extension:
the output, its Lie derivatives along the base system's dynamics, and the derivatives of its virtual inputs ``cover'' the new states.
%
Even though this class is more relaxed than the class of precompensators arrived at through pure prolongations, it retains a key property.
\begin{thm}
  \label{thm:DFL-LSOP}
  If \((\Manifold{M}_2,\) \(\Ideal{I}^{(0)}_2)\) is an extension system of \((\Manifold{M}_1,\) \(\Ideal{I}^{(0)}_1)\) and \(\Dimension \Manifold{M}_2\) \(-\) \(\Dimension \Manifold{M}_1\) \(\geq\) \(2,\) then there exists a regular nonlinear control system \((\Manifold{M}_\bullet,\) \(\Ideal{I}^{(0)}_\bullet)\) where the diagram of extension systems below commutes.
  \begin{center}
    \begin{tikzpicture}[x = 1cm, y = 1cm]

      \node[at = {(0, 0)}] (M2) {\((\Manifold{M}_2, \Ideal{I}^{(0)}_2)\)};
      \node[below = 1 of M2] (Mp) {\((\Manifold{M}_\bullet, \Ideal{I}^{(0)}_\bullet)\)};
      \node[right = 1 of Mp] (M1) {\((\Manifold{M}_1, \Ideal{I}^{(0)}_1)\)};
      
      \draw[- Latex]
        (M2) -- node[above right, midway] {\(\pi\)}  (M1);

      \draw[- Latex]
        (M2) -- node[right, midway] {\(\pi_\bullet\)} (Mp);

      \draw[- Latex]
        (Mp) -- node[above, midway] {\(\pi_1\)} (M1);

    \end{tikzpicture}
  \end{center}
\end{thm}
\begin{pf}
  Since \((\Manifold{M}_2,\) \(\Ideal{I}_2^{(0)})\) is an extension system of \((\Manifold{M}_1,\) \(\Ideal{I}_1^{(0)}),\) there exists a regular zero dynamics foliation \(\Foliation{Z}_1\) of type \((\rho, \kappa)\) on \(\Manifold{M}_1\) where \(\mathfrak{Z}(\pi^* \mathfrak{I}(\Foliation{Z}_1))\) produces the new states.
  There also exists an output \(h: \Manifold{M}_1 \to \Reals^{\rho}\) of vector relative degree \(\kappa\) on \(\Manifold{M}_1\) whose level sets are leaves of \(\Foliation{Z}_1.\)
  Necessarily, the output \(\pi^* h: \Manifold{M}_2 \to \Reals^{\rho}\) yields a vector relative degree \(\geq \kappa\) on \(\Manifold{M}_2.\)
  This difference must be at least two.
  Pick at least one component \(h^j\) of the output with relative degree \(\kappa^j\) where \(\pi^* h^j\) has a larger relative degree \(k > \kappa^j.\)
  Define \(\Manifold{M}_\bullet\) \(\defineas\) \(\Manifold{M}_1\) \(\times\) \(\Reals\) and set, by an abuse of notation,
  \[
    \Ideal{I}_\bullet^{(0)}
      \defineas 
        \langle \Ideal{I}_1{(0)}, \LieOperator_f^{\kappa^j} dh^j - q\,dt \rangle,
  \]
  where we've introduced a new input variable \(q.\)

  The projection map \(\pi_1\) is determined by picking \emph{any} input \(u^i\) that appears non-singularly in \(\LieOperator_f^{\kappa^j} dh^j.\)
  Without loss of generality pick \(u^1\) and set
  \[
    \pi_1:
      (t, q, u^2, \ldots, u^m, x, \LieOperator_f^{\kappa^j} dh^j)
      \mapsto 
      (t, u^1, \ldots, u^m, x).
  \]
  To construct \(\pi_\bullet: \Manifold{M}_2 \to \Manifold{M}_\bullet,\) we observe that \(\LieOperator_f^{\kappa^j} dh^j\) is a function of the states and inputs of \(\Manifold{M}_1\) and thus can be pulled back onto \(\Manifold{M}_2\) to define part of \(\pi_\bullet.\)
  The remaining states \(x\) on \(\Manifold{M}_\bullet\) that come from \(\Manifold{M}_1\) are also similarly covered.
  It suffices to show that \(q\) can be written as a smooth function on \(\Manifold{M}_2.\)
  This must be the case, however, since \(h^j\) has relative degree on \(\Manifold{M}_2\) and has a relative degree \(k\) which is at least greater than, or equal to, the relative degree of \(\pi_1^* h^j.\)
  Thus, \(\pi_\bullet\) exists and is smooth.\qed\hfill
\end{pf}

A corollary of Theorem~\ref{thm:DFL-LSOP} is that any feedback linearizable system arrived at through an extension system can be feedback linearized through a sequence of one-step dynamic extensions.
The original system is therefore in the LSOP class defined by~\cite{nicolau_dynamic_2025}.
The extension system's foliation is the ``natural'' choice for finding the required dynamic extension to the feedback linearizable system.
It is not known a-priori, however, what this natural choice is!
A search is required to find it.

\section{Dynamic Programming on Categories}
\label{section:dynamic-programming}
\cite{bellman_applications_1954} introduced dynamic programming through its applications to decision theory, optimal control theory and mathematical programming.
The approach builds on a key principle: the principle of optimality.
In this section, we reformulate this well-known principle into the language of categories.
\begin{prin}[Principle of Optimality]
  \label{principle:bellman}
  An optimal policy has the property that, whatever the initial state and initial decisions are, the remaining decisions must constitute an optimal policy with regard to the state resulting from the first decisions.
\end{prin}

\subsection{The Optimal Arrow}
A category \(\Category{X}\) is defined by a set of objects \(\ObjectClass_{\Category{X}}\) and a set of arrows \(\ArrowClass_{\Category{X}}\) that start and end at objects.
The arrows must be composable and their composition must be associative.
The identity arrow brings an object to itself.
Composite arrows are arrows that are composed of other non-identity arrows.
Primitive arrows are not composite.
Fix objects \(a,\) \(b \in \Category{X}.\)
The set of arrows between \(a\) and \(b\) is denoted \(\HomSet(a,b)\) and the set of primitive arrows between \(a\) and \(b\) is denoted \(\PrimSet(a, b).\)

Let objects in \(\Category{X}\) denote states of an optimization problem and arrows in \(\Category{X}\) denote transitions between these states.
A loss function \(L: \Category{X} \to \Reals\) over the category assigns a numeric value to every arrow in the category and must assign zero cost to the identity arrow.
Together, the category \(\Category{X}\) and its loss \(L\) sets up an optimization space, allowing us to consider the following problem.
%
%
%
\begin{prob}[The Optimal Arrow]
  \label{prob:optimal-arrow}
  Let \(\Category{G} \subseteq \Category{X}\) and fix \(b \in \Category{X}.\)
  Find, if possible, an object \(a \in \Category{G}\) and arrow \(h^* \in \HomSet(a, b)\) that solves,
  \[
    L(h^*) = 
      \min_{%
        \begin{subarray}{c}
          a \in \Category{G}\\ 
          h \in \HomSet(a, b)
        \end{subarray}%
      } L(h).
  \]
\end{prob}
In this problem, the goal is to find the object \(a\) in the goal category \(\Category{G}\) that has the minimum arrow to our fixed object \(b.\)
An optimal arrow \(h^*\) may not exist even when \(\HomSet(a, b)\) is non-empty:
the minimum may not be well-defined.
For simplicity, assume the minimum always exists when the set of arrows \(\HomSet(a, b)\) is non-empty.
Under this assumption, the notion of optimality is well-defined.
\begin{defn}
  Fix \(a,\) \(b \in \Category{X}.\)
  An arrow \(f: a\to b\) is said to be \defining{optimal} if it minimizes the loss \(L\) over all arrows \(g \in \HomSet(a, b)\) between \(a\) and \(b.\)
  Briefly put,
  \[
    L(f) = \min_{g \in \HomSet(a, b)} L(g).
  \]
  The arrow is said to be \defining{non-optimal} otherwise.
\end{defn}

Denote the set of all optimal arrows from \(a\) to \(b \in \Category{X}\) by \(\OptSet_L(a, b).\)
Dynamic programming algorithms incrementally solve the problem \(\OptSet(a, b)\) for objects \(a\) in the category until the goal is reached.
Principle~\ref{principle:bellman} can be restated as demanding that optimal arrows are post-composed of optimal arrows.
\begin{prin}[The Principle of Optimality]
  \label{principle:optimality}
  Fix \(a,\) \(b \in \Category{X}.\)
  If \(f \in \OptSet(a, b)\) is a \emph{composite} optimal arrow that decomposes into \(f = h \circ g\) for some \(g \in \HomSet(a, c),\) \(h \in \HomSet(c, b)\) and \(c \in \Category{X},\)
  then \(h \in \OptSet(c, b)\) is an optimal arrow.
\end{prin}
The special case on composite arrows is eliminated through logical manipulation as seen in Theorem~\ref{thm:PO-Character}.
\begin{thm}
  \label{thm:PO-Character}
  Principle~\ref{principle:optimality} is equivalent to asking: if \(h \in \ArrowClass_\Category{X}\) is non-optimal, then all arrows of the form \(h \circ g\) for some arrow \(g \in \ArrowClass_\Category{X}\) are non-optimal.
\end{thm}
Theorem~\ref{thm:PO-Character} suggests that Principle~\ref{principle:optimality} is, without uniqueness, a method of \emph{pruning}:
the stripping of non-optimal arrows when building larger optimal arrows.

\subsection{Dynamic Programming}
Principle~\ref{principle:optimality} yields an initial formulation of the dynamic programming equation,
\begin{equation}
  \label{eqn:DP}
  \min_{f \in \HomSet(a,b)} \Loss(f) 
    =
      \min_{
        \begin{subarray}{c}
          c   \in \mathcal{X}\\
          g   \in \HomSet(a, c)\\
          h^* \in \OptSet(c, b)
        \end{subarray}
      }{
        \Loss(h^* \circ g)
      }.
\end{equation}
The dynamic programming equation~\eqref{eqn:DP} says that the optimal loss between objects \(a\) and \(b\in\Category{X}\) can be found by minimizing arrows composed by optimal arrows to the goal \(b.\)
The recursion typically associated with the dynamic programming equation is not apparent in~\eqref{eqn:DP}.
Split the minimization between the primitives to the goal and the remaining arrows to arrive at a formulation that reflects the recursive nature of dynamic programming:
\begin{equation}
  \label{eqn:DP-Recursive}
  \scalebox{0.8}{\mbox{\ensuremath{\displaystyle%
    \min_{f \in \HomSet(a,b)} \Loss(f) 
      =
        \min_{}\{
          \min_{f_p \in \PrimSet(a,b)}{\Loss(f_p)},
          \min_{
            \begin{subarray}{c}
              c   \in \mathcal{X} \setminus \{a, b\}\\
              g   \in \HomSet(a, c)\\
              h^* \in \OptSet(c, b)
            \end{subarray}
          }{
            \Loss(h^* \circ g)
          }
        \}}}}.
\end{equation}
Two subproblems arise in~\eqref{eqn:DP-Recursive}:
a ``smaller'' minimization problem over intermediate objects, and the minimization of primitive ``shortcuts.''
This is the principle of conditional optimization~\citep{sniedovich_dynamic_1992}.

Up until now we have been assuming the principle holds.
It turns out that under a reasonable assumption, the principle does hold.
Recall that \(L: \Category{X} \to \Reals\) is a functor\footnote{Endow \(\Reals\) with the additive category.} if \(L(h \circ g) = L(h) + L(g).\)
Functoriality of the loss is a natural assumption to make:
the loss of an action is the sum of the losses of its more primitive actions.
\begin{thm}
  \label{thm:FunctorialLoss}
  If the loss \(L: \Category{X} \to \Reals\) is a functor, then Principle~\ref{principle:optimality} holds.
\end{thm}
\begin{pf}
  Fix \(c,\) \(b \in \Category{X}.\)
  Let \(h \notin \OptSet(c, b)\) be a non-optimal arrow.
  If there exists an optimal arrow, pick \(h^* \in \OptSet(c, b).\)
  Let \(a \in \Category{X}\) be any object and consider any arrow \(g \in \HomSet{a, c}.\)
  Verify that,
  \[
    L(h \circ g) = L(h) + L(g) > L(h^*) + L(g) = L(h^* \circ g),
  \]
  and conclude that any arrow post-composed by \(h\) is non-optimal.
  This was done for all arrows (when an optimal arrow exists) thus establishing Theorem~\ref{thm:PO-Character}.
  Principle~\ref{principle:optimality} follows.\qed\hfill
\end{pf}

Assume the loss \(L\) is a functor.
This assumption, alongside the assumption that all composite arrows decompose into primitive arrows, simplifies~\eqref{eqn:DP-Recursive} into,
\begin{equation}
  \label{eqn:DP-Functor}
    \min_{f \in \HomSet(a,b)} \Loss(f) 
      =
          \min_{
            \begin{subarray}{c}
              c   \in \mathcal{X} \setminus \{a, b\}\\
              g   \in \PrimSet(a, c)\\
              h^* \in \OptSet(c, b)
            \end{subarray}
          }{
            \Loss(g) + \Loss(h^*)
          }
\end{equation}
Denoting the optimal cost \(\min_{f \in \HomSet(a,b)} \Loss(f)\) by \(J^*_{L}(a, b),\) the dynamic programming equation~\eqref{eqn:DP-Functor} becomes the functional equation,
\begin{equation}
  \label{eqn:DP-Classic}
    J^*_L(a,b)
      =
          \min_{
            \begin{subarray}{c}
              c   \in \mathcal{X} \setminus \{a, b\}\\
              g   \in \PrimSet(a, c)
            \end{subarray}
          }{
            \Loss(g) + J^*_L(c,b)
          }.
\end{equation}
The dynamic programming algorithm uses this expression to incrementally build the optimal arrows to \(b\) from \(a\) using the intermediate optimal costs to \(b\) from an intermediate object \(c.\)

Dynamic programming algorithms produce a sequence of subcategories related by functors to the original category.
At every step, the algorithm prunes non-optimal arrows that terminate on an active object and builds optimal arrows using the composition of the remainder.
The process terminates when at least one object resides in the goal category \(\Category{G}\) and all other arrows are provably non-optimal relative to the arrow from the goal.
In the case where the number of arrows to any given object is finite, search algorithms that solve the shortest-path problem --- like Dijkstra's algorithm~\citep{sniedovich_dijkstras_2006} or \(\text{A}^\star\)~\citep{4082128, martelli_dynamic_1975} --- apply\footnote{These algorithms assume the functoriality of the loss and do not apply in general to any category that satisfies Principle~\ref{principle:optimality}.}.
\begin{center}
  \begin{tikzpicture}[x = 0.9cm, y = 0.9cm]

    \node[at = {(0, 0)}] (A) {b};
    \node[above right = 1 of A] (B) {\(o_1\)};
    \node[right = 1 of A] (C) {c};
    \node[below right = 1 of A] (D) {\(o_2\)};
    
    \node[right = 1 of B] (E) {\(o_3\)};
    \node[right = 1 of C] (F) {a};
    \node[right = 1 of D] (G) {\(o_4\)};
    
    \node[category, fit=(A) (B) (C) (D) (E) (F) (G), ] (X1) {};
    \node[below right=2pt of X1.north west, anchor=north west] {\(\Category{X}\)};
    \node[category, fill=Purple, fill opacity=0.2, fit=(F) (G)] (Goal1) {};
    \node[at = {(Goal1.base)}] {\(\Category{G}\)};

    \draw[- Latex]
      (B) -- (A);
    \draw[- Latex]
      (C) -- (A);
    \draw[- Latex]
      (D) -- (A);
    
    \draw[- Latex]
      (E) -- (B);
    \draw[- Latex]
      (E) -- (C);

    \draw[- Latex]
      (F) -- (B);
    \draw[- Latex]
      (F) -- (C);
    \draw[- Latex]
      (F) -- (D);

    \draw[- Latex]
      (G) -- (C);
    \draw[- Latex]
      (G) -- (D);

    \node[right = 0.6 of X1] (Q) {\(\cdots\)};
    \node[right = 0.75 of Q] (A2) {b};
    \node[above right = 1 of A2] (B2) {\(o_1\)};
    \node[right = 1 of A2] (C2) {c};
    \node[below right = 1 of A2] (D2) {\(o_2\)};
    
    \node[right = 1 of B2] (E2) {\(o_3\)};
    \node[right = 1 of C2] (F2) {a};
    \node[right = 1 of D2] (G2) {\(o_4\)};
    
    \node[category, fit=(A2) (B2) (C2) (D2) (E2) (F2) (G2)] (X2) {};
    \node[below right=2pt of X2.north west, anchor=north west] {\(\Category{X}\)};
    \node[category, fill=Purple, fill opacity=0.2, fit=(F2) (G2)] (Goal2) {};
    \node[at = {(Goal2.base)}] {\(\Category{G}\)};

    \draw[- Latex]
      (B2) -- (A2);
    \draw[- Latex]
      (C2) -- (A2);
    \draw[- Latex]
      (D2) -- (A2);
    \draw[- Latex]
      (E2) -- (A2);
    \draw[- Latex, thick, red]
      (F2) .. controls ($(F2.base)!0.5!(D2.base)$) .. (A2);
    \draw[- Latex, thick, red]
      (G2) -- (A2);

    \draw[- Latex, thick]
      (X2) -- (Q);
    \draw[- Latex, thick]
      (Q)-- (X1);
    
  \end{tikzpicture}
\end{center}

Dynamic programming can be augmented with a cost-to-come heuristic \(H: \Category{X} \to \Reals\) which assigns a cost to an object placing a lower-bound on the optimal cost to arrive to said object.
The heuristic does not change the optimal arrows from a goal category \(\Category{G}\) as long as the heuristic is zero for all objects in \(\Category{G}\) and positive otherwise.
The heuristic's effect on the algorithm is accounted for through the heuristic-adjusted loss,
\[
  L_H(a, b) \defineas L(a, b) - H(a) + H(b).
\]
which is functorial if \(L\) is functorial.
The category must be restricted to only those arrows where the heuristic is non-increasing, i.e. \(-H(a) + H(b) \geq 0.\)
Principle~\ref{principle:optimality} then remains valid.
In the case where the number of arrows to any given object is finite, shortest-path algorithms with a heuristic-adjusted loss are informed by the heuristic.

The category of regular nonlinear control systems is not finite:
the number of arrows into a system is infinite.
The number of arrows to any given system can be made finite through restriction such as when~\cite{levine_differential_2025-1} considers only pure prolongations.
Either Dijkstra's algorithm or \(\text{A}^\star\) apply directly in this case to arrive at a minimal dynamic extension within the restricted class of extensions.
The bounds computed by~\cite{guay_condition_1997} makes this algorithm finite.

\section{Minimal Dynamic Extension through Dynamic Search}
\label{section:deds}
Fix a base, regular nonlinear control system \((\Manifold{M}_0, \Ideal{I}_0^{(0)}).\)
Let \(\Category{X}\) denote the subcategory of regular nonlinear control systems \((\Manifold{M}, \Ideal{I})\) \(\in\) \(\ObjectClass_{\Category{X}}\) which are extension systems of \((\Manifold{M}_0, \Ideal{I}_0^{(0)}).\)
The arrows of this category are the submersions \(\pi \in \ArrowClass_{\Category{X}}\) seen in Definition~\ref{def:extension}.
The primitive arrows of this category are precisely those one-fold dynamic extensions along a regular zero dynamics foliations of type \((1, \kappa)\);
a single channel of length \(\kappa\) is prolonged exactly once to \(\kappa+1\).
Any other extension is composite through the meet operation.
By Theorem~\ref{thm:DFL-LSOP}, all composite arrows decompose into primitive arrows that are extension systems of \((\Manifold{M}_0, \Ideal{I}_0^{(0)})\) in their own right.
Let the goal category \(\Category{G} \subseteq \Category{X}\) be those extension systems that satisfy the feedback linearizability condition~\eqref{eqn:involutivity} and~\eqref{eqn:controllability}.

The minimal dynamic extension to a feedback linearizable system seeks to find the minimal dynamic precompensator in terms of dimension.
Define the loss,
\[
  L(\pi: \Manifold{M}_2 \to \Manifold{M}_1)
    : \Category{X} \to \Reals,
    \pi \mapsto \Dimension\Manifold{M}_2 - \Dimension\Manifold{M}_1,
\]
for any arrow \(\pi \in \ArrowClass_{\Category{X}}\) between a system \((\Manifold{M}_1, \Ideal{I}_1),\) and its extension system \((\Manifold{M}_2, \Ideal{I}_2).\)
The loss measures the change in the system order induced by dynamic extension.
Clearly, \(L(\pi) = 1\) for primitive arrows and is zero for the identity arrow.
The loss is functorial.
It follows by Theorem~\ref{thm:FunctorialLoss} that Principle~\ref{principle:optimality} holds.
The search for a dynamic extension to a feedback linearizable system can be performed by recursively applying~\eqref{eqn:DP-Classic}.
To accelerate the search, we introduce a cost-to-come heuristic.

\subsection{The Leading Integrability Defect Heuristic}
Although we cannot compute the loss between a nonlinear control system \((\Manifold{M}, \Ideal{I}^{(0)})\) and the ``nearest'' feedback linearizable nonlinear control system, we can still define an admissable heuristic that puts a best-case, lower-bound on the loss.
To construct the heuristic, we need a measure of how far away an ideal is from involution.
\begin{defn}[Integrability Defect]
  Let \(\Ideal{I} \subseteq \FormSections{M}\) be an ideal. 
  The \defining{integrability defect} \(\Defect(\Ideal{I})\) is defined as
  \(
    \Defect(\Ideal{I})
      \defineas 
        \Dimension(\Ideal{I} / \Ideal{I}^{(1)}).
  \)
\end{defn}
The \(\Defect(\langle \Ideal{I}^{(k)}, dt \rangle)\) measures the dimension of the space of forms that fail to close under exterior derivative and thus acts as an obstruction to the integrability condition.
This motivates a heuristic.
\begin{defn}[Leading Integrability Defect]
  \hfill\\Let \((\Manifold{M},\) \(\Ideal{I}^{(0)})\) be a regular nonlinear control system.
  The \defining{leading integrability defect} \(\Lid(\Manifold{M}, \Ideal{I}^{(0)})\) is defined in the following way.
  Find the smallest index \(k^\star\) \(\geq\) \(0\) where \(\langle \Ideal{I}^{(k^\star)},\) \(dt \rangle\) fails to be integrable.
  If all the ideals are integrable, set \(k^\star = 0.\)
  Define,
  \(
    \Lid(\Manifold{M}, \Ideal{I}^{(0)})
  \)
  \(\defineas\)
  \(
    \Defect(\langle \Ideal{I}^{(k^\star)}, dt \rangle).
  \)
\end{defn}
The \(\Lid(\cdot)\) is zero (and only zero) for feedback linearizable systems because of~\eqref{eqn:involutivity}.
For dynamically feedback linearizable systems, the \(\Lid(\cdot)\) puts a best-case, lower-bound on how many dynamic extensions must occur to bring the system into linearizability.

Finding an extension that minimizes \(\Lid(\cdot)\) is tricky because, when integrability is achieved on an ideal \(k,\) the defect may spike in the ideal \(k+1.\)
Nevertheless, eliminating the defect in a fixed ideal still often minimizes the leading defect.
Let \(k^\star > 0\) be the index for \(\Lid(\Manifold{M}, \Ideal{I}^{(0)}).\) 
Take \(\sigma \in \langle \Ideal{I}^{k^\star}, dt \rangle\) to be the form that is a wedge product of all the generators of the defective ideal.
The ideal fails~\eqref{eqn:involutivity}, i.e.,
\(
  \langle \Ideal{I}^{k^\star}, dt \rangle
\)
\(\not\subseteq\)
\(
  \langle \Ideal{I}^{k^\star}, dt \rangle^{(\infty)},
\)
and we can identify generators,
\[ 
  \frac{\langle \Ideal{I}^{k^\star}, dt \rangle}{\langle \Ideal{I}^{k^\star}, dt \rangle^{(1)}}
    = \left\langle [\omega^1], \ldots, [\omega^r] \right\rangle.
\]
Necessarily \(d\omega^i \wedge \sigma \neq 0.\)
Find the smallest set \(V = \{dv^1, \ldots, dv^\ell\}\) that solves,
\[
  \begin{array}{cc}
    \min_{%
        V \subseteq \langle \Ideal{I}^{(k^\star - 1)}, dt\rangle        
    } 
      \;& \ell\\
    \text{s.t.}
      \;&
      \sigma \wedge dv^{1} \wedge \cdots \wedge dv^{\ell} \neq 0\\
      \;& 
      d\omega^1 \wedge \sigma \wedge dv^{1} \wedge \cdots \wedge dv^{\ell} = 0\\
      \;&\vdots\\
      \;&
      d\omega^r \wedge \sigma \wedge dv^{1} \wedge \cdots \wedge dv^{\ell} = 0
  \end{array}
\]
The set \(V\) contains the smallest number of exact forms in the preceding ideal (integrable by~\eqref{eqn:involutivity}) that, if prolonged into the defective ideal, renders the \(\omega^i\) differentially closed.
The problem bears resemblance to the hitting set problem.
Solving this problem identifies the outputs \(v^i\) of uniform vector relative degree \(k^\star\) that, upon a one-step dynamic extension, brings the ideal 
\(
  \langle \Ideal{I}^{k^\star}, dt \rangle
\)
into involution.

The aforementioned loss \(L\) can be adjusted with the heuristic \(H \defineas \Lid\) informed by this optimization problem to arrive at a functorial loss \(L_H\) which guides the search for a minimal dynamic extension towards resolving~\eqref{eqn:involutivity}.
Recall that the arrows of the category only need to be restricted to dynamic extensions that do not make the heuristic \emph{worse}:
the restricted category includes extensions that neither improve nor degrade the defect although improvements are favoured by search algorithms.
Let us turn to a few, simple examples to see the heuristic in action.

\subsection{Examples}
The examples below are supported by symbolic computation performed in~\cite{maple_2025}.
Worksheets are available publicly at~\citep{dsouza_deds_2026} alongside a custom \texttt{Maple} package that executes the dynamic search in the finite case.

\subsubsection{A First-Order Dynamic Precompensator}
Define the manifold \(\Manifold{M}\) \(=\) \(\Reals\) \(\times\) \(\Reals^4\) \(\times\) \(\Reals^6\) with input vector \(u \in \Reals^4\) and states \((x,y,\theta,z,w,\beta) \in \Reals^6.\)
Set \(\alpha = \theta + \beta\) and \(\gamma = \alpha + z.\)
Consider the nonlinear control system,
\[
  \scalebox{0.8}{\mbox{\ensuremath{\displaystyle%
    \begin{aligned}
    \Ideal{I}^{(0)}
      = \langle
        & dx - u^1\cos(\gamma) dt,
          dy - u^1\sin(\gamma) dt,
          d\theta - u^2 dt,\\
        & dz - u^3 e^{u^4} \cos(\alpha) dt,
        dw - u^3 e^{u^4} \sin(\alpha) dt,
        d\beta - u^4 e^{u^3} dt
      \rangle.
    \end{aligned}
  }}}
\]
The system is not feedback linearizable.
Moreover, there does not exist a fourth-order dynamic precompensator arrived at through pure prolongations that results in a feedback linearizable system.
A fifth-order dynamic precompensator exists: prolong \(u^2\) once and \((u^3, u^4)\) twice each.
We now show that there exists a first-order dynamic precompensator that brings the system into involution.

The first application of the heuristic-adjusted, dynamic programming equation for this problem involves selecting the amongst the optimal primitive arrow to \((\Manifold{M}, \Ideal{I}^{(0)}).\)
The heuristic-adjusted loss is
\begin{equation}
  \label{eqn:ex:loss}
  L(\pi_1) = 1 - \Lid(\Manifold{M}_2, \Ideal{I}^{(0)}_2) + \Lid(\Manifold{M}, \Ideal{I}^{(0)}).
\end{equation}
The loss of a primitive arrow is minimized by selecting the nonlinear control system that is only one-dimension larger but reduces the leading integrability defect.
In this case, the leading defect occurs at \(\langle \Ideal{I}^{(1)}, dt\rangle\) where the quotient \(\langle \Ideal{I}^{(1)}, dt\rangle\) with \(\langle \Ideal{I}^{(1)}, dt\rangle^{(1)}\) is generated by forms,
\[
  dx - \cot(\gamma) dy,
  dz - \cot(\alpha) dw.
\]
The \(\Lid(\Manifold{M}, \Ideal{I}^{(0)}) = 2.\)
We seek an exact one-form that, upon bringing into this ideal through a one-step dynamic extension, removes the defect.
Prolonging an exact form \(dh\) in \(\langle \Ideal{I}^{(0)}, dt \rangle\) to \(\langle \Ideal{I}^{(1)}, dt\rangle\) (that wasn't there originally) involves extending a virtual input once;
equivalently, this is an extension along a regular zero dynamics foliation induced by output \(h.\)
The exact one-form \(d\gamma\) \(\in\) \(\langle \Ideal{I}^{(0)}, dt \rangle\) is one such candidate.
Using Definition~\ref{def:vector-relative-degree}, determine that \(\gamma\) has relative degree one.
Perform a one-step dynamic extension along its regular zero dynamics foliation to produce the new differential constraint,
\begin{equation}
  \label{eqn:ex:1:constraint}
  \scalebox{0.8}{\mbox{\ensuremath{\displaystyle%
  \begin{aligned}
    d\LieDerivative{f}\gamma - \LieDerivative{f}^2\gamma dt
      &= - u^3 e^{u^4} \sin\alpha d\alpha + du^2\\
      &+ (u^4 e^{u^3} + e^{u^4}\cos\alpha) du^3\\
      &+ (e^{u^3} + u^3 e^{u^4}\cos\alpha) du^4\\
      &- q {dt},
  \end{aligned}
  }}}.
\end{equation}
A virtual input can replace \(u^2\) and be extended to form the new input \(q.\)
This defines a new manifold with additional coordinate \(q\) and the constraint~\eqref{eqn:ex:1:constraint}.
This extended system satisfies the feedback linearizability conditions~\eqref{eqn:involutivity} and~\eqref{eqn:controllability}.
The dynamic programming algorithm terminates since any other extensions necessarily have a larger cost.

\subsubsection{A Second-Order Dynamic Precompensator}
Consider a variation of the previous nonlinear control system,
\[
  \scalebox{0.8}{\mbox{\ensuremath{\displaystyle%
    \begin{aligned}
    \Ideal{I}^{(0)}
      = \langle
        & dx - u^1\cos(\theta + z) dt,
          dy - u^1\sin(\theta + z) dt,
          d\theta - u^2 dt,\\
        & dz - u^3 e^{u^4} \cos(\beta) dt,
        dw - u^3 e^{u^4} \sin(\beta) dt,
        d\beta - u^4 e^{u^3} dt
      \rangle.
    \end{aligned}
  }}}
\]
Not even a fifth-order dynamic precompensator arrived at through pure prolongations yields a feedback linearizable system.
In fact, a second-order dynamic precompensator is required to arrive at integrability.

The same heuristic-adjusted loss~\eqref{eqn:ex:loss} applies but, unlike the previous example, the integrability defect cannot be eliminated fully by a one-step extension.
The obstruction to integrability are the forms,
\[
    dx - \cot(\theta + z) dy, dz - \cot\beta dw,
\]
in the ideal \(\langle \Ideal{I}^{(1)}, dt \rangle.\)
These look similar to the previous example, but, for these forms, one must pick an exact differential from each of the ideals,
\[
    \langle dx, dy, d\theta + dz \rangle, \langle d\beta, dz, dw\rangle,
\]
respectively.
These ideals have no single, common exact one-form.
Any dynamic extension that brings this ideal into involution must involve a two-step extension.
We can begin with an extension of along \(z\) and then subsequently add an extension of \(\theta\) to form an extension along the regular zero dynamics foliation induced by output \((\theta + z, z).\)
The extensions result in a feedback linearizable system.
The dynamic programming algorithm terminates here.

\section{Conclusion}
Facing the inevitable computational wall for dynamic feedback linearization, we presented a dynamic search for the dynamic extension that brings a system into involution.
The approach builds on the structure of regular zero dynamics manifolds to characterize an optimization category of nonlinear control systems.
The search can be made finite via restriction.
Otherwise, a heuristic can guide the search by selecting arrows that are likely to arrive from a feedback linearizable system.
This formulation raises a number of interesting questions.
Is there a way to form equivalence classes of arrows to render the search finite in general?
If not, is it possible to identify the optimal heuristic and systematically generate a sequence of \emph{systems of} linear PDEs through the functional equation?
The presented approach also lends itself to other interesting problems, like identifying a minimal dynamic extension to turn a zero dynamics manifold into a regular zero dynamics manifold (``dynamic'' transverse feedback linearization).
The loss function may also be modified to encode design constraints on the dynamic precompensator.

\bibliography{dflconfbib}

\appendix
\section{Proof of Proposition~\ref{prop:reg-zero-dyn-fol}}
\label{pf:prop:reg-zero-dyn-fol}
\begin{pf}
  First deduce from condition~\ref{prop:reg-zero-dyn-fol:controllability} that \emph{any} first integrals of \(\mathfrak{I}(\Foliation{Z})\) can be made time-independent.
  Moreover, using conditions~\ref{prop:reg-zero-dyn-fol:regularity} and~\ref{prop:reg-zero-dyn-fol:involutivity}, deduce that all the ideals, \(\left\langle \Ideal{I}^{(k)}, dt \right\rangle\) \(\cap\) \(\mathfrak{I}(\Foliation{Z})\) are generated by a finite number of \emph{exact} smooth one-forms, i.e. differentials of smooth functions.
  Define the controllability indices,
  \[
    \begin{aligned}
      \rho^0 &\defineas
        \Dimension(
          \mathfrak{I}(\Foliation{Z})
        )
        -
        \Dimension(
          \langle \Ideal{I}^{(0)}, dt \rangle \cap \mathfrak{I}(\Foliation{Z})
        ),\\
      \rho^i &\defineas
        \Dimension\left(
          \frac{
            \langle \Ideal{I}^{(i-1)}, dt \rangle \cap \mathfrak{I}(\Foliation{Z})
          }{
            \langle \Ideal{I}^{(i)}, dt \rangle \cap \mathfrak{I}(\Foliation{Z})
          }
        \right),
      \quad
        1 \leq i \leq n,
    \end{aligned}
  \]
  and Kronecker indices,
  \[
    \kappa^i \defineas \| \{ j : \rho_j \geq i \} \| - 1.
  \]
  The offset is non-standard and arises due to the inclusion of outputs of relative degree zero.
  These indices are well-defined due to the regularity condition~\ref{prop:reg-zero-dyn-fol:regularity}.

  To show that \(\Foliation{Z}\) is a regular zero dynamics foliation of type \((\rho^0, \kappa)\) where \(\kappa \defineas (\kappa^1, \ldots, \kappa^{\rho^0}),\) we must find the output \(h: \Manifold{M} \to \Reals^{\rho^0}\) of vector relative degree \(\kappa\) whose level sets coincide with the leaves of \(\Foliation{Z}.\)
  The algorithm to determine the outputs mirrors exactly the algorithm in~\cite{dsouza_algorithm_2023} mutatis mutandis.
  We briefly summarize the approach here.

  Moving from the most to the least derived ideal, outputs of higher to lower relative degree may be identified where a change occurs in \(\rho^i.\)
  Starting at \(\langle \Ideal{I}^{(\kappa^{1}-1)}, dt\rangle,\) integrate the differentially closed ideal \(\langle \Ideal{I}^{(\kappa^1-1)}, dt \rangle \cap \mathfrak{I}(\Foliation{Z})\) to find a family of \(\rho^{\kappa^1}\) outputs \(h^1,\) \(\ldots,\) \(h^{\rho^{\kappa^1}}\) with uniform vector relative degree \(\kappa^1.\)
  Compute the regular zero dynamics foliation
  \[
    \Foliation{Z}^{\kappa^1} = \mathfrak{Z}(\langle dh^1, \ldots, dh^{\rho^{\kappa^1}}\rangle).
  \]

  Working backwards, find the next most derived ideal \(j\) where \(\rho^j \neq \rho^{\kappa^1}.\)
  This will be at index \(\rho^{\kappa^2}.\)
  Integrate the ideal \(\langle \Ideal{I}^{(\rho^{\kappa^2}-1)}, dt \rangle \cap \mathfrak{I}(\Foliation{Z})\) and adapt the exact generators so that it takes the form,
  \[
    \langle dh^1, \ldots, \LieOperator_{f}^{\kappa^1 - \rho^{\kappa^2}} dh^{\rho^{\kappa^1}}, dq^1, \ldots, dq^{\rho^{\kappa^2} - \rho^{\kappa^1}}\rangle.
  \]
  Here generators are split to contain the maximal number of generators residing in \(\mathfrak{I}(\Foliation{Z}^{\kappa^1})\) and the remaining novel generators \(dq^i.\)
  This is only possible due to~\ref{prop:reg-zero-dyn-fol:invariance} (i.e.~\eqref{eqn:controlled-invariance}) which implies that repeated Lie derivatives of the exact generators \(dh^i\) found in the derived flag and \(\mathfrak{I}(\Foliation{Z}^{\kappa^1})\) remain in \(\mathfrak{I}(\Foliation{Z}^{\kappa^1})\) \(\subseteq\) \(\mathfrak{I}(\Foliation{Z}).\)
  The remaining generators \(dq^i\) integrate to form an output with uniform vector relative degree \(\kappa^2\) --- they do not reside in the more derived ideal --- and also form vector relative degree with the existing output components \(dh^i.\)
  Thus, we define the new regular zero dynamics foliation,
  \[
    \Foliation{Z}^{\rho^{\kappa^2}}
      \defineas 
      \Foliation{Z}^{\kappa^1}
      \wedge
      \mathfrak{Z}(\langle dq^1, \ldots, dq^{\rho^{\kappa^2} - \rho^{\kappa^1}}\rangle).
  \]
  
  Repeating these series of steps generates a descending sequence of regular zero dynamics foliations,
  \[
    \Foliation{Z}^{\kappa^1} \geq \cdots \geq \Foliation{Z}^{1},
  \]  
  until we arrive to index \(0.\)
  At this index, if \(\rho^0 \neq \rho^1,\) then there exists \(\rho^0 - \rho^1\) outputs of relative degree zero.
  These may be identified by repeating the same series of steps except by directly adapting generators of \(\mathfrak{I}(\Foliation{Z})\) against the regular zero dynamics foliation \(\Foliation{Z}.\)
  This produces the regular zero dynamics foliation \(\Foliation{Z}^0\) which must be equal to \(\Foliation{Z}\) by a dimension counting argument and the fact that all the generators used are from \(\mathfrak{I}(\Foliation{Z}).\)
  Thus, \(\Foliation{Z}\) is a regular zero dynamics foliation.
  \qed\hfill
\end{pf}

\end{document}